\documentclass[12pt]{article}
\usepackage{amsmath,amsopn,amssymb,amsthm,amsfonts}

\newtheorem*{theorem*}{Theorem}

\newtheorem*{claim*}{Claim}

\newtheorem*{lemma*}{Lemma}
\newtheorem*{conj*}{Conjecture}
\newtheorem{conj}{Conjecture}
\newtheorem*{coro}{Corollary}

\theoremstyle{definition}
\newtheorem*{example*}{Example}
\newtheorem{defin}{Definition}

\theoremstyle{remark}

\newtheorem*{remark*}{Remark}

\newcommand{\RR}{\mathbb{R}}
\newcommand{\CC}{\mathbb{C}}

\newcommand{\im}{\textrm{im}\,}

\begin{document}
\title{Mirror K\"{a}hler potential on Calabi--Yau $3$-folds
\footnote{This work was supported in part by the Council of the
Russian Federation President Grants (projects NSh-7256.2010.1 and
MK-842.2011.1).}}
\author{Dmitry V. Egorov}
%\address{
%Ammosov Northeastern federal university, \newline Kulakovskogo str. 48, 677000, Yakutsk, Russia}%
%\email{egorov.dima@gmail.com}%
\date{}

%\thanks{This work was supported in part by the Council of the Russian
%Federation President Grants (projects NSh-7256.2010.1 and
%MK-842.2011.1).}

\maketitle

\sloppy

\begin{abstract}
The classical K\"{a}hler potential  is a real-valued function (KP)
such that one can determine a K\"{a}hler (symplectic) structure by
differentiating KP.  We define a mirror K\"{a}hler potential on
Calabi--Yau $3$-folds, a real-valued function (MKP) such that one
can determine a complex structure by differentiating MKP.

\end{abstract}

\section{Introduction}

The K\"{a}hler potential is a real-valued function $\varphi$ such
that  $dd^c\varphi$ is the K\"{a}hler (symplectic) form. We propose
a mirror construction in complex dimension $3$, namely define a
real-valued function $\varphi$ such that $\varphi$ acted by the
certain symplectic differential operator is a $3$-form that entirely
determines a complex structure. The $3$-forms that determine a
complex structure are described by Hitchin in \cite{Hitchin3f,
Hitchin_stable}.

In \cite{egorov3} we propose a new equation analogous to the complex
Monge--Amp\`{e}re equation studied by Calabi and Yau \cite{Calabi,
Yau}. New equation describes deformation of the complex structure,
whereas the classical case describes deformation of the K\"{a}hler
or symplectic one. It turns out that the mirror K\"{a}hler potential
is a solution of the new equation.

In paper \cite{Hull} the generalized Monge--Amp\`{e}re equation and
the K\"{a}hler potential are defined. However, for the case of the
K\"{a}hler geometry they reduce to the usual Monge--Amp\`{e}re
equation and K\"{a}hler potential respectively.

The author would like to thank  D. Alexeevsky, G.~Cavalcanti and
M.~Verbitsky for useful remarks. The author is grateful to
I.\,A.~Taimanov for support.

\section{Background}
First, note that hereafter we use smooth objects only: functions,
manifolds etc.
\subsection{The K\"{a}hler potential}
Let us briefly recall the definition of the K\"{a}hler potential and
the $dd^c$-lemma.

Let $M$ be a K\"{a}hler manifold. The following statement called
$dd^c$-lemma holds on $M$
$$
\ker d\cap \mathrm{im}\,d^c =  \mathrm{im}\,d\cap \ker d^c =
\mathrm{im}\,dd^c.
$$
The $dd^c$-lemma implies existence of the K\"{a}hler potential, a
real-valued function $\varphi$ such that $dd^c = \omega$, where
$\omega$ is a K\"{a}hler structure. If  manifold is compact, then
the K\"{a}hler form can not be exact and potential can not be a
globally defined function. However, it can be still defined locally.

%There are two aspects of the K\"{a}hler potential.
%
%First one is global. Let $M$ be a  K\"{a}hler manifold with a
%positive closed $(1,1)$-form $\omega$. Then for any
%$\omega'\in[\omega]$ there exists a real-valued function $\varphi$
%on $M$ such that $\omega' = \omega + dd^c\varphi$. In other words,
%the global K\"{a}hler potential determines the deformation of the
%K\"{a}hler structure in its cohomology class.
%
%Second one is local. for any positive closed $(1,1)$-form $\omega$
%there locally exists a real function $\varphi$ such that $dd^c =
%\omega.$

\subsection{Symplectic Hodge theory}\label{section:symplectic}
In this subsection we define
a symplectic differential operator analogous to $d^c$. Let
$(M,\omega)$ be a symplectic manifold of real dimension $2n$. By
$*_s$ denote a symplectic Hodge star. The action of $*_s$ on the
$k$-forms is uniquely determined by the following formula:
$$
\alpha\wedge *_s\beta =
(\omega^{-1})^k(\alpha,\beta)\frac{\omega^n}{n!}.
$$
By definition, put
\begin{equation}\label{delta} d^s =
(-1)^{k+1}\ast_s d\,\ast_s.
\end{equation}
Note that $d^s$ decreases degree of the form by one and  $dd^s =
-d^sd$.
The following statement called $dd^s$-lemma holds on K\"{a}hler
manifolds.
$$
\ker d \cap \im d^s = \im d \cap  \ker d^s   = \im dd^s.
$$

Let us note that the notation of $d^s$ operator is proposed in this
paper. In literature it has various notations: $\delta$,
$d^\Lambda$, $d^\mathcal{J}$ etc. We use current notation due to its
symmetry with the notation of the $d^c$ operator.

For further information on the symplectic Hodge theory see for
example \cite{Tseng} and references therein.

\subsection{Stable forms}
In this subsection we recall the definition of a stable forms by
Hitchin \cite{Hitchin3f,Hitchin_stable}. Let $V$ be a real
$m$-space. A real form $\rho\in \Lambda^pV^*$ is called stable iff
the orbit of $\rho$ under the natural action of $GL(V)$ is open.
This definition can be ported in the clear way  to the case of forms
on manifolds.

If $p=2$ and $m=2n$, then stability is equivalent to the
non-degeneracy condition of  symplectic form. If $p=3$ and $m=6$,
then  any stable form is a real or imaginary part of some
holomorphic volume form.

Here we list some properties of stable $3$-forms.
\begin{itemize}
\item
Stability can be established by some algebraic criterion on
$3$-form.

\item
Any stable {\it real} $3$-form completely determines some almost
complex structure on $V$.

\item For any stable form $\rho$ there exists a dual stable forms
$\hat{\rho}$ such that $\rho+i\hat{\rho}$ is a
$SL(3,\mathbb{C})$-invariant form. If $\rho+i\hat{\rho}$ is closed,
then almost complex structure determined by $\rho$ is integrable and
$\rho+i\hat{\rho}$ is a holomorphic volume form with respect to the
complex structure determined by $\rho$.

\item For any stable form $\rho$ there exists a frame $\{e^i\}$
of $V^*$ such that
\begin{equation}\label{rho}
\rho = e^{135} - e^{245} - e^{146} - e^{236},\quad \hat{\rho}=
e^{235} + e^{145} + e^{136} - e^{246},
\end{equation}
where $e^{ijk} = e^i\wedge e^j\wedge e^k$.
\end{itemize}

%Suppose $V$ carries some orientation.  We say that stable $3$-form
%$\rho$ is {\it positive} iff the orientation induced by the almost
%complex structure of $\rho$ is compatible with the orientation of
%$V$.

\section{The mirror K\"{a}hler potential}
\subsection{Definition} \label{section:def}

Let $M$ be a compact Calabi--Yau $3$-fold with K\"{a}hler form
$\omega$ and holomorphic volume form $\Omega =
\rho+\sqrt{-1}\sigma$. By Darboux' theorem, for any point  $p\in M$
there exists a local co-ordinate chart $\{x^i\}$ in the neighborhood
of $p$ such that locally
\begin{equation}\label{local_omega}
\omega = dx^{12}+dx^{34}+dx^{56},
\end{equation}
where $dx^{ij} =
dx^i\wedge dx^j$. Put
\begin{equation}\label{local_rho1}
\rho_0 = dx^{135} - dx^{245} - dx^{146} - dx^{236},\quad \sigma_0 =
dx^{235} + dx^{145} + dx^{136} - dx^{246},
\end{equation}
or
\begin{equation}\label{local_rho2}
\Omega_0 = \rho_0 + \sqrt{-1}\sigma_0 =
(dx^1+\sqrt{-1}dx^2)\wedge(dx^3+\sqrt{-1}dx^4)\wedge(dx^5+\sqrt{-1}dx^6).
\end{equation}
\begin{defin}
We say that a locally defined real-valued function $\varphi$ is a
{\it local mirror K\"{a}hler potential} iff
$$
\rho = dd^s\varphi\sigma_0,\quad   \sigma = -dd^s\varphi\rho_0
$$
or
$$
\Omega = -\sqrt{-1}dd^s\varphi\:\Omega_0,
$$
where $\Omega_0 = \rho_0+\sqrt{-1}\sigma_0$.
\end{defin}
Since stable forms
completely determine complex structure, one can see that the mirror
K\"{a}hler potential determines complex structure whenever the
symplectic one is given.

As in the classical case one can not define the K\"{a}hler potential
globally unless the holomorphic volume form is exact. Nevertheless,
we give the following defintion.
\begin{defin}
We say that a global function on $M$ is a {\it global mirror
K\"{a}hler potential} iff for $\Omega'\sim\Omega$
$$
\Omega' = \Omega - \sqrt{-1}dd^s\varphi\:\Omega.
$$
\end{defin}

\begin{example*}Suppose $\CC^3$ with a flat Hermitian
metric $\sum dz^i\otimes d\bar{z}^i$; the K\"{a}hler form $\omega =
(\sqrt{-1}/2)\sum dz^i\wedge d\bar{z}^i$ and holomorphic volume form
$\Omega = dz^1\wedge dz^2\wedge dz^3$. Then the following identities
hold:
$$
\omega = dd^c\varphi, \quad \Omega = -
(\sqrt{-1}/3)dd^s\varphi\:\Omega,
$$
where $\varphi = \sum|z^i|^2$.
\end{example*}

\subsection{Does the mirror K\"{a}hler potential exist?}

The question about existence of the local or global mirror
K\"{a}hler potential is open. The possible line of attack could be
using the classical continuity method. Consider for example $\RR^6$
with the  standard Euclidean symplectic structure $\omega_0$ and
compatible complex structure given by $\Omega =
\rho+\sqrt{-1}\sigma$. We have seen in the example above that there
exists a mirror K\"{a}hler potential whenever there is a flat
metric, i.e., $\rho=\rho_0$; $\sigma = \sigma_0$.

Now consider a parametrized system of equations:
\begin{equation}\label{para}
dd^s\varphi_t\rho_0 = -(1-t)\sigma_0 - t\sigma; \quad
dd^s\varphi_t\sigma_0 = (1-t)\rho_0 + t\rho;\quad  t\in[0,1].
\end{equation}
Since the set of almost complex structures compatible with the given
symplectic form is convex, RHS  describes the continuous family of
stable $3$-forms.

Then we can use the continuity method. Let $\mathcal{T}$ be a subset
of $[0,1]$ such that $t\in\mathcal{T}$, whenever parametrized
equation has solution for that $t$. Since $0\in\mathcal{T}$, it is
not empty. Now one has to establish some estimates on the equation
\eqref{para} and prove that $\mathcal{T}$ is open and closed in
$[0,1]$.

We conjecture that (at least) for compact Calabi--Yau $3$-folds
there exist local and global mirror K\"{a}hler potentials.

\begin{conj}Let $M$ be a compact Calabi--Yau $3$-fold with K\"{a}hler form
$\omega$ and holomorphic volume form $\Omega=\rho+\sqrt{-1}\sigma$.
Then for any point $p\in M$ there locally exists a co-ordinate chart
$\{x^i\}$ such that $\omega$ takes form \eqref{local_omega} and a
real-valued function $\varphi$ such that
$$\Omega = -\sqrt{-1}dd^s\varphi\:\Omega_0,
$$
where $\Omega_0$ is defined by the formula \eqref{local_rho2} with
respect to the local chart.

\end{conj}

\begin{conj}\label{conj1}Let $M$ be a compact Calabi--Yau $3$-fold with K\"{a}hler form
$\omega$ and holomorphic volume form $\Omega=\rho+\sqrt{-1}\sigma$.
Then for $\Omega'\sim\Omega$ there exists a global real-valued
function $\varphi$ such that
$$
\Omega' = \Omega - \sqrt{-1}dd^s\varphi\:\Omega.
$$
\end{conj}

\subsection{Functional class}Let us recall the definition of the
plurisubharmonic functions. Assume $\CC^3$ with a symplectic form
$\omega$. Let $\mathcal{C}(\omega)$ be a convex conical set of
alternating $2$-vectors such that for any
$\xi\in\mathcal{C}(\omega)$
$$\omega(\xi)>0.$$
Recall that a real-valued (smooth) function $\varphi$ is called
plurisubharmonic iff for any $\xi\in\mathcal{C}(\omega)$
$$
(dd^c\varphi)(\xi) \geq 0.
$$
Function $\varphi$ is called pluriharmonic iff $\varphi$ and
$-\varphi$ are plurisubharmonic functions.

Now let us state analogous definition concerning special Lagrangian
submanifolds. Let $\rho+\sqrt{-1}\sigma$ be a holomorphic volume
form on $\CC^3$; then  define sets of $3$-vectors
$\mathcal{C}(\rho)$ and $\mathcal{C}(\sigma)$ as above.

\begin{defin}\label{positive} We call a stable $3$-form $\tau$ positive (strictly
positive) modulo stable form $\rho$ and denote by $\tau \geq 0$
($\tau
> 0$) $\mod \rho$ iff for any $\xi\in\mathcal{C}(\rho)$
$$
\tau(\xi) \geq 0 \quad (\tau(\xi) > 0).
$$
We call a stable $3$-form $\tau$ negative (strictly negative) modulo
stable form $\rho$ and denote by $\tau \leq 0$ ($\tau < 0$) $\mod
\rho$ iff $-\tau \geq 0$ ($-\tau > 0$) $\mod \rho$.
\end{defin}

\begin{defin}\label{psh} A real-valued function $\varphi$ is called {\it special
Lagrangian plurisubharmonic} iff for any $\xi\in\mathcal{C}(\rho)$
and $\eta\in\mathcal{C}(\sigma)$
\begin{equation}\label{psh_ineq}
(dd^s\varphi\sigma)(\xi) \geq 0; \quad (dd^s\varphi\rho)(\eta) \leq
0,
\end{equation}
or equivalently
\begin{equation}\label{psh_ineq2}
dd^s\varphi\sigma\geq 0 \mod \rho; \quad dd^s\varphi\rho\leq 0 \mod
\sigma.
\end{equation}
\end{defin}
\begin{defin}
A real-valued function $\varphi$ is called {\it strictly special
Lagrangian plurisubharmonic} iff inequalities \eqref{psh_ineq} or
equivalently \eqref{psh_ineq2} are strict.
\end{defin}
\begin{remark*}
This definition is closely related with the one given in \cite{HL}.
In fact, these two definitions determine the same class of
functions. However, they use a differential operator $d^\phi$
related with the calibration $\phi$. The $d^\phi$ is first defined
in \cite{Verbit} for the case of $G_2$-manifolds.
\end{remark*}

%They prove that the above definition is equivalent to the following.
%
%A real function $\varphi$ defined locally or globally is called {\it
%special Lagrangian pluriharmonic} (respectively {\it special
%Lagrangian plurisubharmonic}) iff $\varphi$ is harmonic
%(respectively subharmonic) on every special Lagrangian submanifold.

%In this subsection define a functional class of the mirror
%K\"{a}hler  potential. Assume oriented $6$-space $V$. We say that a
%stable $3$-form $\rho$ is positive and denote by $\rho
%>0$ iff
%$$\rho\wedge\hat{\rho} =
%\rho\wedge J(\rho)> 0,$$ where $J$ is the almost complex structure
%induced by $\rho$. Note that if $\rho>0$, then $-\rho>0$, but
%$\hat{\rho} < 0$. Thus, we can distinguish stable forms that set an
%almost complex structure compatible with the given orientation.
%
%
%%Now we can define a functional class of the mirror K\"{a}hler
%%potential. Assume $\CC^3$ with flat metric; $\omega = $
%
%Let $\rho$ be a positive stable $3$-form on $\CC^3$. We say that a
%function $\varphi:\CC^3\to\RR$ is {\it symplectic pluriharmonic} iff
%$$
%dd^s\varphi\rho = 0; \quad dd^s\varphi\hat{\rho} = 0.
%$$
%%
%We say that a function $\varphi:\CC^3\to\RR$ is {\it symplectic
%plurisubharmonic} iff
%$$
%dd^s\varphi\rho \leq 0;\quad dd^s\varphi\hat{\rho} \geq 0,
%$$
%where $3$-form $\tau \leq 0$ iff $\tau$ is not positive and
%vice-versa.

\begin{claim*}
Any special Lagrangian plurisubharmonic function is
plurisubharmonic.
\end{claim*}
\begin{proof}
By H\"{o}rmander's definition of $G$-subharmonic functions, where
$G$ is a linear group \cite[Definition 5.1.1]{Hormander_convex}
special Lagrangian plurisubharmonic functions are
$Sp(3,\CC)$-subharmonic. Namely, the set $\mathcal{S}$ of special
Lagrangian plurisubharmonic functions satisfies the following
conditions.
\begin{itemize}
\item $\mathcal{S}$ contains every affine function;
\item $\mathcal{S}$ is invariant under $Sp(3,\CC)$;
\item the weak maximum principle is valid for $\mathcal{S}$;
\item $\mathcal{S}$ is maximal with the preceding properties.

\end{itemize}
By \cite[Theorem 5.1.7]{Hormander_convex}, the set of
$Sp(3,\CC)$-subharmonic functions is a set of plurisubharmonic
functions.

\end{proof}

\begin{coro}
The mirror K\"{a}hler potential is a plurisubharmonic function.
\end{coro}

%If $\omega = \sum_{i=1}^{3}dx^i dy^i$, then $\varphi$ is symplectic
%pluriharmonic iff
%\begin{equation}\label{sph1}
%\frac{\partial^2\varphi}{\partial x^i \partial x^j} -
%\frac{\partial^2\varphi}{\partial y^i \partial y^j} = 0,\quad
%\frac{\partial^2\varphi}{\partial y^i \partial x^j} +
%\frac{\partial^2\varphi}{\partial x^i \partial y^j} = 0;\quad i,j =
%1,2,3;
%\end{equation}
%\begin{equation}\label{sph2}\sum_{i=1}^{3}\frac{\partial^2\varphi}{\partial x^i \partial x^i} =
%0.
%\end{equation}

%Suppose we start from flat $\CC^3$ with symplectic and complex
%structures given by $\omega_0$ and $\Omega_0$ respectively. Let us
%construct new structures
%$$
%\omega_1 = dd^c\varphi, \quad \Omega_1 =
%(\sqrt{-1}/3)dd^s{\varphi}\:\Omega_0.
%$$
%Does there exist a diffeomorphism $\Phi:\CC^3\to\CC^3$  such that
%$\Omega_1 = \Phi^*\Omega_0$ and $\omega_0 = \Phi^*\Omega_1$ ? In
%other words, is there diffeomorphism of  $(\CC^3,\omega_0,\Omega_1)$
%and $(\CC^3,\omega_1,\Omega_0)$ respecting symplectic and complex
%structures?

\section{New equation }

In paper \cite{egorov3} we propose a new equation on the
$3$-dimensional Calabi--Yau metrics and prove the following solution
existence theorem.

\begin{theorem*}
Let $(M,\omega)$ be a compact K\"{a}hler $3$-manifold such that
$c_1(M)=0$; let $\Omega = \rho+i\sigma$ be a holomorphic volume form
on $M$. Then  the following equation on the unknown real $3$-forms
$\alpha$ and $\beta$
\begin{equation}\label{new1}
(\rho+dd^s\alpha)\wedge(\sigma +dd^s\beta) = e^F\rho\wedge\sigma
\end{equation}
or equivalently
\begin{equation}\label{new2}(\Omega + dd^s\psi)\wedge (\bar{\Omega} + dd^s\bar{\psi}) =
e^F\Omega\wedge\bar{\Omega},\quad \psi = \alpha+\sqrt{-1}\beta
\end{equation}
has solution provided that
\begin{enumerate}
\item $\rho+dd^s\alpha$ and $\sigma +dd^s\beta$ are primitive stable forms;
\item $\sigma +dd^s\beta$ is dual to $\rho+dd^s\alpha$ in the sense of the stable forms;
\item real function $F$ is normalized: $\int_M{e^F\rho\wedge\sigma} = \int_M
\rho\wedge\sigma$.
\end{enumerate}
\end{theorem*}

If the global mirror K\"{a}hler potential $\varphi$ exists, then
equations \eqref{new1} and \eqref{new2} take form
\begin{equation}\label{newnew1}
(\rho+dd^s\varphi\sigma)\wedge(\sigma -dd^s\varphi\rho) =
e^F\rho\wedge\sigma
\end{equation}
and
\begin{equation}\label{newnew2}
(\Omega - \sqrt{-1}dd^s\varphi\Omega)\wedge (\bar{\Omega} +
\sqrt{-1}dd^s\varphi\bar{\Omega}) = e^F\Omega\wedge\bar{\Omega}
\end{equation}
respectively.

\begin{conj}
Let $M$ be a compact Calabi--Yau $3$-fold with K\"{a}hler form
$\omega$ and holomorphic volume form $\Omega=\rho+\sqrt{-1}\sigma$.
Then equation \eqref{newnew1} or equivalently \eqref{newnew2} has
unique solution: a real-valued function $\varphi$ provided that
\begin{enumerate}
\item $\rho+dd^s\varphi\sigma > 0 \mod \rho$ and $\sigma - dd^s\varphi\rho > 0 \mod \sigma$,
where positivity is in the sense of the Definition \ref{positive};
\item $\int_M{e^F}\omega^n = \int_M{\omega^n}$;
\item $\int_M{\varphi}\,\omega^n = 0$.
\end{enumerate}

\end{conj}
Obviosly, this conjecture follows from the Conjecture \ref{conj1}.

Now let us write down the local form of the equation
\eqref{newnew1}. First, recall the local form of the
Monge--Amp\`{e}re equation.

Suppose $U$ is an open set of $M$ with co-ordinate chart $\{z^i\}$
such that $\Omega = dz^1\wedge dz^2\wedge dz^3$ on $U$; then the
global equation $(\omega + dd^c\varphi)^n = e^F\omega^n$ is locally
equivalent to the following equation on $U$
$$\det
\varphi_{i\bar{j}} = \mathrm{const},
$$ where  $\varphi_{i\bar{j}} =
\partial^2\varphi/\partial z^i\partial \bar{z}^j$.

Now suppose $U$ is an open set of $M$ with co-ordinate chart
$\{x^i\}$ such that $\omega = dx^{12}+dx^{34}+dx^{56}$ on $U$; then
the global equation \eqref{newnew1} is locally equivalent to the
following equation on $U$
\begin{equation}\label{new_explicit}
\begin{array}{l}
\quad(\varphi_{22}+\varphi_{33}+\varphi_{55})(\varphi_{11}+\varphi_{44}+\varphi_{66})\\
+\ (\varphi_{11}+\varphi_{44}+\varphi_{55})(\varphi_{22}+\varphi_{33}+\varphi_{66})\\
+\ (\varphi_{11}+\varphi_{33}+\varphi_{66})(\varphi_{22}+\varphi_{44}+\varphi_{55})\\
+\ (\varphi_{22}+\varphi_{44}+\varphi_{66})(\varphi_{11}+\varphi_{33}+\varphi_{55})\\
-\ (\varphi_{12}+\varphi_{34}+\varphi_{56})^2 -\
(-\varphi_{12}-\varphi_{34}+\varphi_{56})^2
-(\varphi_{12}-\varphi_{34}-\varphi_{56})^2\\
-\ (-\varphi_{12}+\varphi_{34}-\varphi_{56})^2
-2\left[(\varphi_{13}-\varphi_{24})^2
+(\varphi_{36}+\varphi_{45})^2+(\varphi_{15}-\varphi_{26})^2\right.\\
+\ \left.(\varphi_{16}+\varphi_{25})^2+(\varphi_{35}-\varphi_{46})^2
+(\varphi_{14}+\varphi_{23})^2\right] = \mathrm{const},
\end{array}\end{equation}
where  $\varphi_{ij} =
\partial^2\varphi/\partial x^i\partial x^j$.

To investigate equation \eqref{new_explicit} consider the case when
the Calabi--Yau manifold is fibered by flat special Lagrangian tori
\cite{syz} and perform analysis similar to \cite{Leung}. Suppose
manifold is semi-flat and the mirror K\"{a}hler potential depends on
the odd co-ordinates only. Then equation \eqref{new_explicit} takes
the following form.
\begin{equation}\label{new_explicit_semi}
\begin{array}{l}
\quad\varphi_{11}\varphi_{33}+\varphi_{11}\varphi_{55} +
\varphi_{33}\varphi_{55} -\varphi_{13}^2
-\varphi_{15}^2-\varphi_{35}^2  = \mathrm{const},
\end{array}\end{equation}
where  $\varphi_{ij} =
\partial^2\varphi/\partial x^i\partial x^j$.
Obviously, the equation \eqref{new_explicit_semi} is the sum of
three real Monge--Amp\`{e}re equations. It is known that for any
solution of the real Monge--Amp\`{e}re equation one can produce a
new solution by performing the Legendre transformation. Therefore,
by performing the partial Legendre transformation on the mirror
K\"{a}hler potential, we obtain new solutions of
\eqref{new_explicit} just as in the classical case. Note that the
partial Legendre transfomation of a plurisubharmonic function is
again plurisubharmonic \cite{Hormander}.

Conjecturally,  equation \eqref{new_explicit} is also related to the
some type of the Monge--Amp\`{e}re equations \cite{lychagin}. This
is an open question.

\section{Conclusion}
In conclusion we state some open problems.

1) It is natural to consider higher dimensional cases, though the
holomorphic volume form is not stable if $n>3$. The uniqueness of
the $n=3$ case is that a single real $3$-form determines complex
structure. In higher dimensions one needs a pair of real $n$-forms
to determine a complex structure.

2)  What is a relation between the classical and mirror K\"{a}hler
potentials? What is a relation between the complex Monge--Amp\`{e}re
equation and the new equation?

\end{document}